\documentclass{amsart}
\usepackage{amsmath,amsfonts,amssymb}
\usepackage{amsthm}
\usepackage{ifthen}
\usepackage{longtable}
\usepackage{array}
\usepackage{url}
\usepackage{color}

\newcommand{\Z}{\mathbb{Z}}
\newcommand{\Q}{\mathbb{Q}}

\newcommand{\C}{\mathbb{C}}

\DeclareMathOperator{\Gal}{Gal}

\newtheorem*{theorem*}{Theorem}
\newtheorem{theorem}{Theorem}

\newtheorem{problem}{Problem}

\theoremstyle{remark}

\begin{document}
\title[Irrationality]{Irrationality of growth constants associated with polynomial recursions}

\author[S. Wagner]{Stephan Wagner}
\address{S. Wagner,
Department of Mathematical Sciences, 
Stellenbosch University,
Private Bag X1,
Matieland 7602, South Africa
and
Department of Mathematics,
Uppsala Universitet,
Box 480,
751 06 Uppsala,
Sweden
}
\email{swagner\char'100sun.ac.za,stephan.wagner\char'100math.uu.se}

\author[V. Ziegler]{Volker Ziegler}
\address{V. Ziegler,
University of Salzburg,
Hellbrunnerstrasse 34/I,
A-5020 Salzburg, Austria}
\email{volker.ziegler\char'100sbg.ac.at}

\begin{abstract}
We consider integer sequences that satisfy a recursion of the form $x_{n+1} = P(x_n)$ for some polynomial $P$ of degree $d > 1$. If such a sequence tends to infinity, then it satisfies an asymptotic formula of the form $x_n \sim A \alpha^{d^n}$, but little can be said about the constant $\alpha$. In this paper, we show that $\alpha$ is always irrational or an integer. In fact, we prove a stronger statement: if a sequence $G_n$ satisfies an asymptotic formula of the form $G_n = A \alpha^n + B + O(\alpha^{-\epsilon n})$, where $A,B$ are algebraic and $\alpha > 1$, and the sequence contains infinitely many integers, then $\alpha$ is irrational or an integer.
\end{abstract}

\maketitle

\section{Introduction}

Integer sequences obtained by polynomial iteration, i.e., sequences that satisfy a recursion of the form
$$x_{n+1} = P(x_n),$$
occur in several areas of mathematics. Several interesting examples can be found in Finch's book \cite{Finch:2003} on mathematical constants, Chapter 6.10.

\medskip

Let us give two concrete examples: the first is the sequence given by $x_0 = 0$ and $x_{n+1} = x_n^2 + 1$ for $n \geq 0$, which is entry A003095 in the On-Line Encyclopedia of Integer Sequences (OEIS, \cite{OEIS}). Among other things, $x_n$ is the number of binary trees whose height (greatest distance from the root to a leaf) is less than $n$. This sequence grows very rapidly: there exists a constant $\beta \approx 1.2259024435$ (the digits are A076949 in the OEIS) such that $x_n = \lfloor \beta^{2^n} \rfloor$. However, this formula is not an efficient way to compute the elements of the sequence, since one needs the whole sequence to evaluate the constant $\beta$ numerically: it is given by
$$\beta = \prod_{n=1}^{\infty} \big(1 + x_n^{-2} \big)^{2^{-n-1}}.$$
Another well-known example is Sylvester's sequence (A000058 in the OEIS), which is given by $y_0 = 2$ and $y_{n+1} = y_n^2-y_n+1$. It arises in the context of Egyptian fractions, $y_n$ being the smallest positive integer for each $n$ such that
$$\frac{1}{y_0} + \frac{1}{y_1} + \frac{1}{y_2} + \cdots + \frac{1}{y_n} < 1.$$ 
Again, there is a pseudo-explicit formula for the sequence: for a constant $\gamma \approx 1.5979102180$, we have $y_n = \lfloor \gamma^{2^n} + \frac12 \rfloor$. However, $\gamma$ can again only be expressed in terms of the elements of the sequence. This is also the reason why little is known about the constants $\beta$ and $\gamma$ in these two examples and generally growth constants associated with similar sequences that satisfy a polynomial recursion.

\medskip

In this short note, we will prove that the constants $\beta$ and $\gamma$ in these examples are---perhaps unsurprisingly---irrational, as are all growth constants associated with similar sequences that follow a polynomial recursion. The precise statement reads as follows:

\begin{theorem}\label{thm:main_recseq}
Suppose that an integer sequence satisfies a recursion of the form $x_{n+1} = P(x_n)$ for some polynomial $P$ of degree $d > 1$ with rational coefficients. Assume further that $x_n \to \infty$ as $n \to \infty$. Set
$$\alpha = \lim_{n \to \infty} \sqrt[d^n]{x_n}.$$
Then $\alpha$ is a real number greater than $1$ that is either irrational or an integer. 
\end{theorem}

It is natural to conjecture that the constants $\beta$ and $\gamma$ in our first two examples are not only irrational, but even transcendental. This is not always true for polynomial recurrences in general, though: for example, consider the sequence given by $z_1 = 3$ and $z_{n+1} = z_n^2 - 2$. It is not difficult to prove that
$$z_n = L_{2^n} = \Big( \frac{1 + \sqrt{5}}{2} \Big)^{2^n} + \Big( \frac{1 - \sqrt{5}}{2} \Big)^{2^n}$$
for all $n \geq 1$, where $L_{n}$ denotes the $n$-th Lucas number. Thus the constant $\alpha$ in Theorem~\ref{thm:main_recseq} would be the golden ratio in this example.

\medskip

In the following section, we briefly review the classical method to determine the asymptotic behaviour of polynomially recurrent sequences. Theorem~\ref{thm:main_recseq} will follow as a consequence of a somewhat stronger result, Theorem~\ref{th:irr_crit}. This theorem and its proof, which makes use of the subspace theorem, will be given in Section~\ref{sec:subspace}.

\section{Asymptotic formulas for polynomially recurrent sequences}\label{sec:asymp}

There is a classical technique for the analysis of polynomial recursions. A treatment of the two examples given in the introduction
can already been found in the 1973 paper of Aho and Sloane \cite{AhoSloane:1973} (along with many other examples). See also Chapter 2.2.3 of the book of Greene and Knuth~\cite{GreeneKnuth:1990} for a discussion of the method.

Let the polynomial $P$ in the recursion $x_{n+1} = P(x_n)$ be given by
$$P(x) = c_d x^d + c_{d-1} x^{d-1} + \cdots + c_0.$$
Note that
$$P(x) = c_d \Big( x + \frac{c_{d-1}}{d c_d} \Big)^d + O(x^{d-2}).$$
So if we perform the substitution $y_n = c_d^{1/(d-1)} (x_n + \frac{c_{d-1}}{d c_d})$, the recursion becomes
\begin{align*}
y_{n+1} &= c_d^{1/(d-1)} \Big(P(x_n) + \frac{c_{d-1}}{d c_d} \Big)\\
&= c_d^{d/(d-1)} \Big( x_n + \frac{c_{d-1}}{d c_d} \Big)^d + O(x_n^{d-2}) \\
&= y_n^d + O(y_n^{d-2}).
\end{align*}
Let us assume that $x_n \to \infty$, thus also $y_n \to \infty$. It is easy to see that $x_n$ and $y_n$ are increasing from some point onwards in this case. We can also assume, without loss of generality, that none of the $y_n$ is zero: if not, we simply choose a later starting point. Taking the logarithm, we obtain
$$\log y_{n+1} = d \log y_n + O(y_n^{-2})$$
or equivalently
\begin{equation}\label{eq:O-est}
\log \Big( \frac{y_{n+1}}{y_n^d} \Big)  = O(y_n^{-2}).
\end{equation}
Next express $\log y_n$ as follows:
\begin{align*}
\log y_n &= d \log y_{n-1} + \log \Big( \frac{y_n}{y_{n-1}^d} \Big) =
d^2 \log y_{n-2} + d \log \Big( \frac{y_{n-1}}{y_{n-2}^d} \Big) + \log \Big( \frac{y_n}{y_{n-1}^d} \Big) \\
&= \cdots = d^n \log y_0 + \sum_{k=0}^{n-1} d^{n-k-1} \log \Big( \frac{y_{k+1}}{y_k^d} \Big).
\end{align*}
Extending to an infinite sum (which converges since $\log (y_{k+1}/y_k^d)$ is bounded) yields
$$\log y_n = d^n \Big( \log y_0 + \sum_{k=0}^{\infty} d^{-k-1} \log \Big( \frac{y_{k+1}}{y_k^d} \Big) \Big) - \sum_{k=n}^{\infty} d^{n-k-1} \log \Big( \frac{y_{k+1}}{y_k^d} \Big).$$
Set
$$\log \alpha = \log y_0 + \sum_{k=0}^{\infty} d^{-k-1} \log \Big( \frac{y_{k+1}}{y_k^d} \Big),$$
so that
$$\log y_n = d^n \log \alpha - \sum_{k=n}^{\infty} d^{n-k-1} \log \Big( \frac{y_{k+1}}{y_k^d} \Big).$$
In view of~\eqref{eq:O-est} and the fact that $y_n \leq y_{n+1} \leq \cdots$ for sufficiently large $n$, this gives
$$\log y_n = d^n \log \alpha + O(y_n^{-2}),$$
and thus finally
$$y_n = \alpha^{d^n} + O\big(\alpha^{-d^n}\big).$$
This means that
$$x_n = c_d^{-1/(d-1)} \alpha^{d^n} - \frac{c_{d-1}}{d c_d} + O\big(\alpha^{-d^n}\big).$$

\section{Application of the subspace theorem}\label{sec:subspace}

We now combine the asymptotic formula from the previous section with an application of the subspace theorem to prove our main result on polynomial recursions. In fact, we first state and prove a somewhat stronger result that implies Theorem~\ref{thm:main_recseq}.

\begin{theorem}\label{th:irr_crit}
Assume that the sequence $G_n$ attains an integral value infinitely often, and that it satisfies an asymptotic formula of the form
 $$G_n=A \alpha^n+B+O(\alpha^{-\epsilon n}),$$
where $\alpha > 1$, $A$ and $B$ are algebraic numbers with $A \neq 0$, $\epsilon > 0$, and the constant implied by the $O$-term does not depend on $n$. Then the number $\alpha$ is either irrational or an integer. 
\end{theorem}

In order to prove the irrationality of $\alpha$ we make use of the following version of the subspace theorem, which is most suitable for our purposes (cf. \cite[Chapter V, Theorem 1D]{Schmidt:1991}).

\begin{theorem}[Subspace theorem]
Let $K$ be an algebraic number field and let
$S \subset M(K)=\{\text{canonical absolute values of}\; K\}$ be a finite set of absolute
values which contains all of the Archimedian ones. For each $\nu \in S$
let $L_{\nu,1}, \ldots, L_{\nu,N}$ be $N$ linearly independent linear
forms in $n$ variables with coefficients in $K$. Then for given
$\delta>0$, the solutions of the inequality
\[\prod_{\nu \in S} \prod_{i=1}^N |L_{\nu ,i}(\mathbf x) |_{\nu}^{n_\nu} <
\overline{|\mathbf x|}^{-\delta}\]
with $\mathbf x = (x_1,x_2,\ldots,x_N) \in \mathfrak a_K^N$ ($\mathfrak a_K$ being the maximal order of $K$) and $\mathbf x \neq \mathbf 0$, where
$|\cdot|_{\nu}$ denotes the valuation corresponding to $\nu$, $n_{\nu}$ is the
local degree, and
\[\overline{|\mathbf x|}= \max_{1 \leq i \leq N \atop 1 \leq j
\leq \deg K}|x_i^{(j)}|,\]
the maximum being taken over all conjugates $x_i^{(j)}$ of all entries $x_i$ of $\mathbf x$, lie in finitely many proper subspaces of $K^N$.
\end{theorem}

\begin{proof}[Proof of Theorem~\ref{th:irr_crit}]
Let us assume contrary to the statement of Theorem \ref{th:irr_crit} that $\alpha=p/q$ is rational, where $p$ and $q$ are coprime positive integers, $p > q$, and $q \neq 1$. Moreover, assume that their prime factorizations are 
$$p=p_1^{n_1} \dots p_k^{n_k} \quad \text{and} \quad q=q_1^{m_1} \dots q_{\ell}^{m_\ell}.$$
We choose $K$ in the subspace theorem to be the normal closure of $\Q(A,B)$ and write $D=[K:\Q]$. We fix one embedding of $K$ into $\C$ so that we can assume that $K\subseteq \C$. Moreover, let us write $A$ and $B$ as $A=\beta_1/Q$ and $B=\beta_2/Q$, where $\beta_1$ and $\beta_2$ lie in the maximal order $\mathfrak a_K$ of $K$ and $Q$ is a positive integer such that the ideals $(\beta_1,\beta_2)$ and $(Q)$ are coprime.

If $n$ is an index such that $G_n$ is an integer we deduce that there exists an algebraic integer $a$ which may depend on $n$ such that
\begin{equation}\label{eq:G_integral}
 G_n=\frac{\beta_1 p^n+\beta_2 q^n+a}{Q q^n}=A \alpha^n+B+\frac{a}{Q q^n}=A\alpha^n+B+O(\alpha^{-\epsilon n}).
\end{equation}
Since we are assuming that $G_n$ is a rational integer, we can write the algebraic integer $a$ in the form $a=X-\beta_1 p^n-\beta_2q^n$, with $X\in \Z$.
Moreover, we know that 
\begin{equation}\label{eq:a_small}
 |a|<CQ\left(\frac{q^{1+\epsilon}}{p^{\epsilon}}\right)^n,
\end{equation}
where $C$ is the constant implied by the $O$-term.

Assume that $K$ has signature $(r,s)$. We choose
$$S=\{\infty_1,\dots,\infty_{r+s},\mathfrak p_{1,1},\dots,\mathfrak p_{k,t_k},\mathfrak q_{1,1},\dots,\mathfrak q_{t,u_{\ell}}\},$$
where the valuations $\mathfrak p_{i,1},\dots,\mathfrak p_{i,t_i}$ are all valuations lying above $p_i$ for $1\leq i \leq k$ and the valuations $\mathfrak q_{j,1},\dots,\mathfrak q_{j,u_j}$ are all valuations lying above $q_j$ for $1\leq j \leq \ell$. Moreover, let
$$\Gal(K/\Q)=\{\sigma_1,\dots,\sigma_r,\sigma_{r+1},\bar\sigma_{r+1},\dots,\sigma_{r+s},\bar\sigma_{r+s}\},$$
so that the valuation $\infty_i$ is given by $|x|_{\infty_i}=|\sigma_i^{-1} x|$, where $|\cdot|$ is the usual absolute value of $\C$. Finally, denote the conjugates of $\beta_1$ and $\beta_2$ by $\beta_j^{(i)} = \sigma_i(\beta_j)$. We have the formula
$$|x_1\beta_1^{(i)}+x_2\beta_2^{(i)}+x_3|_{\infty_i}=|x_1\beta_1+x_2\beta_2+x_3|$$
for arbitrary rational numbers $x_1,x_2,x_3$.

Next, we construct suitable linear forms to apply the subspace theorem. Let us write
$x_1=p^n$, $x_2=q^n$ and $x_3=a$, thus $N=3$. We choose our linear forms as $L_{\nu,1}(\mathbf x)=x_1, L_{\nu,2}(\mathbf x)=x_2$ for all $\nu\in S$ and $L_{\nu,3}(\mathbf x)=x_3$ if $\nu$ lies above one of the valuations $p_1,\dots,p_k$. We choose $L_{\nu,3}(\mathbf x)=\beta_1 x_1+\beta_2 x_2+x_3$ if $\nu$ lies above one of the valuations $q_1,\dots,q_\ell$. Finally if $\nu=\infty_i$ then we put  $L_{\infty_i,3}(\mathbf x)=(\beta_1-\beta_1^{(i)}) x_1+(\beta_2-\beta_2^{(i)}) x_2+x_3$.

Using the product formula (cf. \cite[pages 99--100]{Lang:ANT}) and trivial estimates we obtain
\begin{align*}
 \prod_{\nu\in S} |L_{\nu,1}(\mathbf x)|_\nu^{n_\nu}&=1,&
 \prod_{\nu\in S} |L_{\nu,2}(\mathbf x)|_\nu^{n_\nu}&=1,\\
 \prod_{\nu|q_j \atop 1\leq j\leq \ell} |L_{\nu,3}(\mathbf x)|_\nu^{n_\nu}&\leq q^{-Dn},&
 \prod_{\nu|p_i \atop 1\leq i\leq k} |L_{\nu,3}(\mathbf x)|_\nu^{n_\nu}&\leq 1.
\end{align*}
Thus we are left to compute the quantities $|L_{\infty_i,3}(\mathbf x)|_{\infty_i}$. We obtain
\begin{align*}
|L_{\infty_i,3}(\mathbf x)|_{\infty_i} &=|(\beta_1-\beta_1^{(i)}) x_1+(\beta_2-\beta_2^{(i)}) x_2+x_3 |_{\infty_i}\\
&= |\beta_1 p^n+\beta_2 q^n+a-\beta_1^{(i)}p^n-\beta_2^{(i)}q^n|_{\infty_i}\\
&=|X-\beta_1^{(i)}p^n-\beta_2^{(i)}q^n|_{\infty_i}\\
&=|X-\beta_1 p^n-\beta_2 q^n|=|a|.
\end{align*}
Combining all inequalities, we have
\begin{equation}\label{eq:prod_bound}
 \prod_{\nu \in S} \prod_{i=1}^3 |L_{\nu ,i}(\mathbf x) |_{\nu}^{n_\nu}\leq q^{-Dn}|a|^D< (CQ)^D\left(\frac{q}{p}\right)^{\epsilon Dn}.
\end{equation}
Now choose $\delta>0$ small enough so that
$$\left(\frac{q}{p}\right)^{\epsilon D}<p^{-\delta}.$$
In view of~\eqref{eq:a_small}, the inequality $|a|_\nu \leq p^n$ holds for all valuations $\nu$ lying above $\infty$ for sufficiently large $n$, so that $\overline{|\mathbf x|} = |x_1| = p^n$. Hence we obtain
$$(CQ)^D\left(\frac{q}{p}\right)^{\epsilon Dn}<(p^n)^{-\delta}=\overline{|\mathbf x|}^{-\delta}$$
for sufficiently large $n$. In view of \eqref{eq:prod_bound} we have shown that
\begin{equation}\label{eq:ST}
\prod_{\nu \in S} \prod_{i=1}^n |L_{\nu ,i}(\mathbf x) |_{\nu}^{n_\nu} <
\overline{|\mathbf x|}^{-\delta}.
\end{equation}

By the subspace theorem all solutions $x_1,x_2$ and $x_3$ to \eqref{eq:ST} lie in finitely many subspaces of $K^3$. Since by assumption there are infinitely many solutions, there exists one subspace $T\subseteq K^3$ which contains infinitely many solutions. Let $T$ be defined by
$t_1x_1+t_2x_2+t_3x_3=0$, with fixed algebraic integers $t_1,t_2,t_3\in \mathfrak a_K$. Then there must be infinitely many integers $n$ such that $t_1p^n+t_2q^n+t_3a=0$ which is in contradiction to \eqref{eq:a_small} and the assumption that $p > q > 1$. Thus we can conclude that $\alpha$ cannot be rational, unless $q=1$ so that $\alpha$ is an integer.

\end{proof}

Now the proof of Theorem~\ref{thm:main_recseq} is straightforward.

\begin{proof}[Proof of Theorem~\ref{thm:main_recseq}]
As derived in Section~\ref{sec:asymp}, if an integer sequence satisfies a recursion of the form $x_{n+1} = P(x_n)$ for some polynomial $P$ of degree $d > 1$ with rational coefficients and $x_n \to \infty$ as $n \to \infty$, then an asymptotic formula
of the form
$$x_n = A \alpha^{d^n} + B + O(\alpha^{-d^n})$$
holds. If $\alpha$ is rational, but not an integer, then we have an immediate contradiction with Theorem~\ref{th:irr_crit}.
\end{proof}

\section{Further generalizations}

Let us remark that Theorem \ref{th:irr_crit} can be extended to number fields:

\begin{theorem}\label{th:general-irr-crit}
Let $L$ be a number field and $\mathfrak a_L$ its maximal order.
 Assume that the sequence $G_n$ attains values in $\mathcal O_L$ infinitely often, and that it satisfies an asymptotic formula of the form
 $$G_n=A \alpha^n+B+O(|\alpha|^{-\epsilon n}),$$
where $|\alpha| > 1$, $A$ and $B$ are algebraic numbers with $A \neq 0$, $\epsilon > 0$, and the constant implied by the $O$-term does not depend on $n$. Then the number $\alpha$ is either an algebraic integer in $\mathfrak a_L$ or $\alpha\not\in L$.
\end{theorem}

The proof of this theorem is similar to the proof of Theorem \ref{th:irr_crit}. In particular let $K$ be the normal closure of $L(A,B)$ and assume that $\alpha=p/q$ with $p\in \mathfrak a_L$ and $q\in \Z$ with $q>1$. Then we consider the prime ideal factorizations 
$$(p)=\mathfrak p_1^{n_1} \dots \mathfrak p_k^{n_k} \quad \text{and} \quad (q)=\mathfrak q_1^{m_1} \dots \mathfrak q_{\ell}^{m_\ell}$$
in $K$. We can construct the same linear forms as in the proof of Theorem \ref{th:irr_crit} and use the subspace theorem to get a contradiction. 

It is also possible to consider a higher-dimensional variant of Theorem \ref{thm:main_recseq}. Let $f_1,\dots,f_N \in \Z[X_1,\dots,X_N]$ be polynomials of degree $d$. Then we can consider a sequence $(\mathbf x_n)$ with $\mathbf x_n=\left(x^{(1)}_n,\dots,x^{(N)}_n\right)\in \Z^N$ for all $n\geq 0$ satisfying the polynomial recurrence
$$\mathbf x_{n+1}=f(\mathbf x_n)=(f_1(\mathbf x_n),\dots,f_N(\mathbf x_n)).$$
With this notation at hand we pose the following problem:

\begin{problem}
 Assume that $\max \left\{x^{(1)}_n,\dots,x^{(N)}_n\right\} \rightarrow \infty$ as $n\rightarrow \infty$, and let
 $$\alpha=\lim_{n\rightarrow \infty} \sqrt[d^n]{\max\left\{x^{(1)}_n,\dots,x^{(N)}_n\right\}}.$$ 
 Is $\alpha$ necessarily irrational or an integer?
\end{problem}

\end{document}